\documentclass{amsart}

\usepackage{amsfonts}
\usepackage{amssymb}

\numberwithin{equation}{section}

\theoremstyle{plain}

\newtheorem{theorem}[subsection]{Theorem}
\newtheorem{corollary}[subsection]{Corollary}
\newtheorem{proposition}[subsection]{Proposition}
\newtheorem{lemma}[subsection]{Lemma}
\newtheorem*{claim}{Claim}

\theoremstyle{definition}

\newtheorem*{problem}{Problem}
\newtheorem*{example}{Example}

\renewcommand{\leq}{\leqslant}
\renewcommand{\geq}{\geqslant}

\providecommand{\supp}{\mathop{\rm supp}\nolimits}

\providecommand{\Prog}{\mathop{\rm Prog}\nolimits}

\providecommand{\Spec}{\mathop{\rm Spec}\nolimits}

\newcommand{\wh}{\widehat}
\newcommand{\E}{\mathbb{E}}
\newcommand{\Z}{\mathbb{Z}}
\newcommand{\N}{\mathbb{N}}

\newcommand{\C}{\mathbb{C}}

\begin{document}

\title{Chowla's cosine problem}

\author{Tom Sanders}
\address{Department of Pure Mathematics and Mathematical Statistics\\
University of Cambridge\\
Wilberforce Road\\
Cambridge CB3 0WA\\
England } \email{t.sanders@dpmms.cam.ac.uk}

\begin{abstract}
Suppose that $G$ is a discrete abelian group and $A \subset G$ is a finite symmetric set. We show two main results. 
\begin{enumerate}
\item Either there is a set $\mathcal{H}$ of $O(\log^c|A|)$ subgroups of $G$ with $|A \triangle \bigcup \mathcal{H}| = o(|A|)$ where $\bigcup \mathcal{H}=\bigcup_{H \in \mathcal{H}}{H}$, or there is a character $\gamma \in \wh{G}$ such that $-\wh{1_A}(\gamma) =\Omega(\log^c|A|)$ where $c>0$ is the same absolute constant.
\item If $G$ is finite and $|A| = \Omega(|G|)$ then either there is a subgroup $H \leq G$ such that $|A \triangle H| = o(|A|)$, or there is a character $\gamma \in \wh{G}$ such that $-\wh{1_A}(\gamma) =\Omega(|A|^{\Omega(1)})$.
\end{enumerate}
\end{abstract}

\maketitle

\section{Introduction}

Suppose that $G$ is an abelian group, which we shall think of as discrete. We write $\wh{G}$ for the dual group, that is the compact abelian group of homomorphisms $\gamma:G \rightarrow S^1$ and denote the Haar probability measure on $\wh{G}$ by $\mu$. In general, if $S$ is a compact open subset of a locally compact abelian group then $\mu_S$ denotes Haar measure restricted to $S$ and normalized to be a probability measure.

The Fourier transform is the map $\wh{.}:\ell^1(G) \rightarrow L^\infty(\wh{G})$ which takes $f \in \ell^1(G)$ to $\wh{f}$ determined by
\begin{equation*}
\wh{f}(\gamma):=\sum_{x \in G}{f(x)\overline{\gamma(x)}}.
\end{equation*}
It is useful to use the Fourier transform to define the space $A(G)$ of functions $f \in \ell^1(G)$ endowed with the norm
\begin{equation*}
\|f\|_{A(G)}:=\|\wh{f}\|_{L^1(\wh{G})}=\int{|\wh{f}(\gamma)|d\gamma},
\end{equation*}
where the integration is, of course, with respect to the Haar probability measure on $\wh{G}$. We have the Fourier inversion formula, Plancherel's theorem and Parseval's theorem which we use liberally and without further mention; the classic text \cite{WR} of Rudin includes all the details. 

Now, suppose that $A$ is a finite symmetric subset of $G$. Since $A$ is symmetric it is easy to see that $\wh{1_A}$ is real-valued, and it becomes natural to ask how positive or negative it may get. The former question is trivial: a quick calculation reveals that $\wh{1_A}(0_{\wh{G}})=|A|$, and since trivially $\|\wh{1_A}\|_{L^\infty(\wh{G})} \leq |A|$, we see that $\wh{1_A}$ gets as large as it possibly could. The latter is not so simple: writing
\begin{equation*}
M_G(A):=\sup_{\gamma \in \wh{G}}{-\wh{1_A}(\gamma)},
\end{equation*}
we want a lower bound on $M_G(A)$ in terms of $|A|$.

In the paper \cite{SC}, Chowla asked for such a lower bound on $M_{\Z}(A)$. By a simple averaging argument  and the Littlewood conjecture (resolved independently by Konyagin \cite{SVK} and McGehee, Pigno and Smith \cite{MPS}) one gets that $M_\Z(A) = \Omega(\log |A|)$. Prior to the resolution of the Littlewood conjecture some progress had been made for generic $A$ by Roth in \cite{KFRCos}. However, the logarithmic barrier was first breached by Bourgain in \cite{JBCos}, using a method which was later refined by Ruzsa, \cite{IZRCos}, to give the following result.
\begin{theorem}[{\cite[Theorem 2]{IZRCos}}]\label{thm.ruzsastheorem}
Suppose that $A$ is a finite, non-empty symmetric set of integers. Then
\begin{equation*}
M_{\Z}(A) = \exp(\Omega(\sqrt{\log |A|})).
\end{equation*}
\end{theorem}
This theorem is the best known to date; in the other direction there are sets $A$ with $M_{\Z}(A) = O(\sqrt{|A|})$, but nothing better.

Recently, Green and Konyagin in \cite{BJGSVK} began work extending Littlewood's conjecture to abelian groups other than $\Z$. In both Littlewood's conjecture and Chowla's problem there is a simple obstacle to the most obvious extension: if $H$ is a finite subgroup of $G$ then $\|1_H\|_{A(G)}=1$ and $M_G(H)=0$. It turns out that in a number of cases this is really the only barrier.

In their work Green and Konyagin addressed the discrete analogue of the Littlewood conjecture and their result can be used in the same way as the Littlewood conjecture in $\Z$ to get that $M_{\Z/p\Z}(A) = \log^{\Omega(1)} |A|$ if $|A|=(p+1)/2$. We are able to do somewhat better -- even than the obvious analogue of Theorem \ref{thm.ruzsastheorem} -- and shall show the following.
\begin{theorem}\label{thm.primetheorem}
Suppose that $p$ is a prime and $A \subset \Z/p\Z$ is symmetric and of size $(p+1)/2$. Then 
\begin{equation*}
M_{\Z/p\Z}(A) = \Omega(p^{\Omega(1)}).
\end{equation*}
\end{theorem}
In fact, through careful accounting one can arrange for the $\Omega(1)$ constant to be $1/3$. For comparison Spencer showed in \cite{JS} that there are sets $A \subset \Z/p\Z$ of size $(p+1)/2$ with $\sup_{\gamma \neq 0}{|\wh{1_A}(\gamma)|} = O(p^{1/2})$ and hence $M_{\Z/p\Z}(A) = O(p^{1/2})$. It seems interesting to try to close this gap.

As is turns out we shall prove the following generalization of the above result. To preserve the strength of the bound we impose the additional constraint that $A$ has density bounded away from $0$.
\begin{theorem}\label{thm.maintheorem}
Suppose that $G$ is a finite abelian group and $A \subset G$ is symmetric with $|A| = \Omega(|G|)$. Then there is a subgroup $H \leq G$ such that
\begin{equation*}
M_G(A) = |A \triangle H|^{\Omega(1)}.
\end{equation*}
\end{theorem}
Note that if $p$ is prime then the only subgroups of $\Z/p\Z$ are trivial, whence if $|A|=(p+1)/2$ we se that $|A \triangle H| = \Omega(p)$ and we have Theorem \ref{thm.primetheorem}.

Now, for comparison, if $A$ is the union of a (large) finite subgroup $H$ and $K$ other points, then $M_G(A) \leq K$ and $\min_{H' \leq G}{|A \triangle H'|} =K$, so the result is best possible up to the power. In fact, in many cases this has to be smaller than $1$ as noted for $\Z/p\Z$ above.

The main defect of the above theorem is that $A$ is required to have density bounded away from $0$ and so it has no bearing on Chowla's original problem. Our next result recovers the situation although at considerable cost to the bound.
\begin{theorem}\label{thm.densityindependentmaintheorem}
Suppose that $G$ is an abelian group and $A \subset G$ is a non-empty symmetric set. Then there is a set $\mathcal{H}$ of subgroups of $G$ with $|\mathcal{H}| =O(M_G(A))$ such that
\begin{equation*}
M_G(A) = \log^{\Omega(1)} |A \triangle \bigcup{\mathcal{H}}|,
\end{equation*}
where $\bigcup{\mathcal{H}}=\bigcup_{H \in \mathcal{H}}{H}$.
\end{theorem}
Note that we have to allow unions of subgroups. Consider, for example, the case of $H\cup H'$ where $H$ and $H'$ are subgroups with  $H \cap H'=\{0_G\}$. It is easy to show that $M_G(H \cup H') \leq 1$.

We close this introduction with an outline of the paper. In \S\ref{sec.trivial} we illustrate some trivial arguments for showing when $M_G(A)$ is non-zero; these trivial arguments turn out to be central to our later work. Then, in \S\S\ref{sec.boolspec}\verb!&!\ref{sec.maintheorem} we prove Theorem \ref{thm.maintheorem}, through some analysis of the spectrum of boolean functions.

The remainder of the paper is then devoted to proving Theorem \ref{thm.densityindependentmaintheorem} (which uses Theorem \ref{thm.maintheorem}) in \S\S\ref{sec.bs}--\ref{sec.maintheorem2}. To do this we recall the technology of Bourgain systems from \cite{BJGTS2}, although this is entirely contained in \S\S\ref{sec.bs}\verb!&!\ref{sec.tools} and may be treated as a black box from the perspective of the rest of the paper. Finally, \S\ref{sec.conrem} closes with some concluding remarks.

\section{A trivial estimate}\label{sec.trivial}

It is instructive for us to begin by proving a weak version of our main results:
\begin{proposition}\label{prop.vsimple}
Suppose that $G$ is an abelian group and $A \subset G$ is a non-empty symmetric set. Then there is a subgroup $H \leq G$ such that
\begin{equation*}
M_G(A) \geq \frac{1}{2}1_{\{x>0\}}(|A \triangle H|) = \begin{cases}\frac{1}{2} & \textrm{ if } A \neq H;\\ 0 & \textrm{ if }A = H.\end{cases}.\end{equation*}
\end{proposition}
The method of proof in fact gives somewhat more general results which will be needed later: it applies not just to functions which are boolean, but also those which are \emph{almost} boolean, and it is stronger for functions which are constant on cosets of a large subgroup.

Suppose that $G$ is an abelian group, $p \in [1,\infty]$ and $f \in \ell^1(G)$. Then $f$ is said to be \emph{$(\epsilon,p)$-almost boolean} if
\begin{equation*}
\inf_{A \subset G}{\|f-1_A\|_{\ell^p(G)}} \leq \epsilon \|f\|_{\ell^p(G)}.
\end{equation*}
In the following the reader may wish to specialize to the case $V=\{0_G\}$ and $f=1_A$, where $\epsilon=0$, which gives Proposition \ref{prop.vsimple}.
\begin{lemma}\label{lem.trivialdensitytheorem}
Suppose that $G$ is an abelian group with a finite subgroup $V$ and $f \in \ell^1(G)$ is a symmetric $(\epsilon,1)$-almost boolean function, constant on cosets of $V$, with $f(x_0)>1/2$ for some $x_0 \in G$. Then at least one of the following holds:
\begin{enumerate}
\item there is a subgroup $H \leq G$ such that $\|f-1_H\|_{\ell^1(G)} \leq \epsilon \|f\|_{\ell^1(G)}$;
\item or there is a character $\gamma \in \wh{G}$ such that $\wh{f}(\gamma) \leq -|V|/8 + \epsilon \|f\|_{\ell^1(G)}$.
\end{enumerate}
\end{lemma}
\begin{proof}
Let $H:=\{x: f(x) >1/2\}$, and note that by definition
\begin{equation*}
|f(x)-1_H(x)| \leq \min\{|f(x)-1|, |f(x)|\} \textrm{ for all } x \in G,
\end{equation*}
so by integrating we get
\begin{equation}\label{eqn.approximationoffbyH}
\|f-1_H\|_{\ell^1(G)} \leq \inf_{A \subset G}{\|f-1_A\|_{\ell^1(G)}} \leq \epsilon \|f\|_{\ell^1(G)},
\end{equation}
since $f$ is $(\epsilon,1)$-almost boolean.

$f$ is symmetric, so $H$ is symmetric and $f(x_0)>1/2$, so $H$ is non-empty. Thus, either $H$ is a group, or there are elements $x,y \in H$ such that $x+y \not \in H$. Now, $f$ is constant on cosets of $V$, so $H$ is constant on cosets of $V$, and since cosets of $V$ partition $G$ we conclude that $x+V,y+V \subset H$ and $(x+y+V) \cap H = \emptyset$.

In view of all this information we evaluate the inner product
\begin{equation*}
\langle 1_H, (\mu_V-(\mu_{x+V}+\mu_{-x+V})/2) \ast (\mu_V-(\mu_{y+V}+\mu_{-y+V})/2) \rangle.
\end{equation*}
On the one hand, by symmetry of $H$, this is
\begin{eqnarray*}
1_H \ast \mu_V(0_G) - 1_H \ast \mu_V(x) &-&1_H \ast \mu_V(y)\\& + &1_H\ast \mu_V(x-y)/2 + 1_H \ast \mu_V(x+y)/2,
\end{eqnarray*}
which is at most $1-1-1+1/2+0=-1/2$, from our various assumptions on $x$ and $y$. On the other, by Plancherel's theorem, the inner product is equal to
\begin{equation*}
\int_{\gamma \in V^\perp}{\wh{1_H}(\gamma
)(1-\Re \gamma(x))(1-\Re \gamma(y))d\gamma} \geq 4\inf_{\gamma \in \wh{G}}{\wh{1_H}(\gamma)}\mu(V^\perp).
\end{equation*}
The integral is well defined because the range of integration is restricted to $V^\perp$ and so the value of $\gamma(x)$ is independent of the coset representative of $x+V$ that is chosen; the inequality follows since $\gamma$ maps into $S^1$, whence $|\Re \gamma| \leq 1$ which implies that $2 \geq 1-\Re\gamma \geq 0$. We conclude that there is some character $\gamma \in \wh{G}$ such that
\begin{equation*}
\wh{1_H}(\gamma) \leq -\mu(V^\perp)^{-1}/8 = -|V|/8,
\end{equation*}
and we are in the second case by (\ref{eqn.approximationoffbyH}), after noting that the symmetry of $f$ implies that $\wh{f}$ is real-valued.
\end{proof}
A very similar argument yields another result which we shall need later, this time for $(\epsilon,\infty)$-almost boolean functions. Again, if we specialize to the case $V=\{0_G\}$ and $f=1_A$, where $\epsilon=0$ we get Proposition \ref{prop.vsimple}.
\begin{lemma}\label{lem.trivialdensity2}
Suppose that $G$ is an abelian group with a finite subgroup $V$ and $f\in \ell^1(G)$ is a real symmetric $(\epsilon,\infty)$-almost boolean function, constant on cosets of $V$, with $f(x_0)>1/2$ for some $x_0 \in G$. Then at least one of the following holds:
\begin{enumerate}
\item the set $H:=\{x \in G: f(x)>1/2\}$ is a subgroup of $G$;
\item or there is a character $\gamma \in \wh{G}$ such that $\wh{f}(\gamma) \leq -|V|(1/8-5\epsilon \|f\|_{\ell^\infty(G)}/8) $.
\end{enumerate}
\end{lemma}
\begin{proof}
We proceed as in the previous lemma. $f$ is symmetric, so $H$ is symmetric and $f(x_0)>1/2$, so $H$ is non-empty. Thus, either $H$ is a group, or there are elements $x,y \in H$ such that $x+y \not \in H$. Now, $f$ is constant on cosets of $V$, so $H$ is constant on cosets of $V$, and since cosets of $V$ partition $G$ we conclude that $x+V,y+V \subset H$ and $(x+y+V) \cap H = \emptyset$.

In view of all this information we evaluate the inner product
\begin{equation*}
\langle f, (\mu_V-(\mu_{x+V}+\mu_{-x+V})/2) \ast (\mu_V-(\mu_{y+V}+\mu_{-y+V})/2) \rangle.
\end{equation*}
On the one hand by symmetry of $f$, this is
\begin{equation*}
f \ast \mu_V(0_G) - f \ast \mu_V(x)-f \ast \mu_V(y) + f \ast \mu_V(x-y)/2 + f \ast \mu_V(x+y)/2,
\end{equation*}
which is at most 
\begin{equation*}
1 - (1-\epsilon\|f\|_{\ell^\infty(G)}) - (1-\epsilon\|f\|_{\ell^\infty(G)}) + 1/2 + \epsilon\|f\|_{\ell^\infty(G)}/2,
\end{equation*}
equals $-1/2 +5\epsilon\|f\|_{\ell^\infty(G)}/2$. On the other, by Plancherel's theorem, the inner product is equal to
\begin{equation*}
\int_{\gamma \in V^\perp}{\wh{f}(\gamma
)(1-\Re \gamma(x))(1-\Re \gamma(y))d\gamma} \geq 4\inf_{\gamma \in \wh{G}}{\wh{f}(\gamma)}\mu(V^\perp).
\end{equation*}
The integral is well defined because the range of integration is restricted to $V^\perp$ and so the value of $\gamma(x)$ is independent of the coset representative of $x+V$ that is chosen; the inequality follows since $\gamma$ maps into $S^1$, whence $|\Re \gamma| \leq 1$ which implies that $2\geq 1-\Re\gamma \geq 0$. We conclude that there is some character $\gamma \in \wh{G}$ such that
\begin{equation*}
\wh{f}(\gamma) \leq (-1/8 +5\epsilon\|f\|_{\ell^\infty(G)}/8)\mu(V^\perp)^{-1} = -|V|(1/8 -5\epsilon\|f\|_{\ell^\infty(G)}/8),
\end{equation*}
after noting that the symmetry of $f$ implies that $\wh{f}$ is real-valued. The lemma follows.
\end{proof}

\section{The spectrum of boolean functions}\label{sec.boolspec}

Suppose that $G$ is an abelian group, $A$ is a finite subset of $G$ and $\epsilon \in (0,1)$ is a parameter. Then the \emph{spectrum} of $A$ is defined to be
\begin{equation*}
\Spec_\epsilon(A):=\{\gamma \in \wh{G}: |\wh{1_A}(\gamma)| \geq \epsilon |A|\}.
\end{equation*}
We shall be considering both powers and convolution powers and it will be useful to have a notation for the latter. For a function $f:\wh{G} \rightarrow \C$ we write $f^{(r)}$ for the $r$-fold convolution of $f$ with itself, so $f^{(1)}=f$ and $f^{(r+1)}=f \ast f^{(r)}$.

As first observed by Bourgain in \cite{JBCos}, the fact that $1_A$ is boolean gives us considerable information about the spectrum of $A$ in the following sense. For any positive integer $r$ we have $1_A^r=1_A$, whence
\begin{equation}\label{eqn.encodeconvolution}
\wh{1_A}^{(r)}(\gamma)= \wh{1_A}(\gamma) \textrm{ for all } \gamma \in \wh{G}.
\end{equation}
This is only useful because in our problem we are able to assume that $\|1_A\|_{A(G)}$ is small, whence a small amount of $\ell^p(\wh{G})$ information about $\wh{1_A}$ can be leveraged into much more $\ell^p(\wh{G})$ information.

We use these data in both the lemmas of this section; in the first we show that if $M_G(A)$ is small then we either have considerable structure of the spectrum or there is a character at which the Fourier transform is large but not too large.
\begin{lemma}\label{lem.technicallemma}
Suppose that $G$ is a finite abelian group, $A$ is a symmetric set of density $\alpha>0$ and $\epsilon \in (0,1)$ is a parameter. Then at least one of the following is true:
\begin{enumerate}
\item $\Spec_\epsilon(A)$ is a subgroup of $\wh{G}$;
\item we have the estimate $M_G(A) \geq \epsilon^2\alpha|A|/2\|1_A\|_{A(G)}$ (where the $A$ in $1_A$ is not to be confused with that in the algebra norm $\|\cdot\|_{A(G)}$);
\item or there is a character $\gamma \in \wh{G}$ such that $\epsilon|A| > \wh{1_A}(\gamma) \geq \epsilon^2\alpha|A|/2$.
\end{enumerate}
\end{lemma}
\begin{proof}
$\Spec_{\epsilon}(A)$ is clearly a symmetric neighborhood of $0_{\wh{G}}$, meaning that $\Spec_{\epsilon}(A)$ is a symmetric set containing $0_{\wh{G}}$. Hence it is a group iff for all $\gamma',\gamma'' \in \Spec_{\epsilon}(A)$ we have $\gamma' + \gamma'' \in \Spec_{\epsilon}(A)$. Thus we are either in the first case of the lemma, or else there are characters $\gamma', \gamma''  \in \Spec_{\epsilon}(A)$ with $\wh{1_A}(\gamma'+\gamma'') < \epsilon |A|$; suppose we are given such characters.

The $r=2$ instance of equation (\ref{eqn.encodeconvolution}), and the fact that in our normalization Haar measure on $\wh{G}$ assigns to each element of $\wh{G}$ the mass $|G|^{-1}$, tells us that
\begin{eqnarray*}
\wh{1_A}(\gamma'+\gamma'') &=&\int{\wh{1_A}(\gamma'-\gamma)\wh{1_A}(\gamma''+\gamma)d\gamma}\\ & \geq & |G|^{-1}\wh{1_A}(\gamma')\wh{1_A}(\gamma'') - M_G(A) \|1_A\|_{A(G)}.
\end{eqnarray*}
Since $\wh{1_A}(\gamma'), \wh{1_A}(\gamma'') \geq \epsilon |A|$ we conclude that either $M_G(A) \geq \epsilon^2\alpha |A|/2\|1_A\|_{A(G)}$ and we are in the second case of the lemma, or else $\wh{1_A}(\gamma'+\gamma'') \geq \epsilon^2\alpha |A|/2$, and we are in the third case of the lemma with $\gamma=\gamma'+\gamma''$.
\end{proof}
The following lemma makes use of (\ref{eqn.encodeconvolution}) for larger values of $r$ and is really the heart of Theorem \ref{thm.maintheorem}.  The basic idea is fairly straightforward and we explain it now in words.

Suppose that $\gamma$ is such that $\wh{1_A}(\gamma)=\epsilon |A|$ is large and that $\wh{1_A}\geq 0$ for a contradiction.  This positivity gives
\begin{equation*}
(\wh{1_A}- \epsilon |A| 1_{\{\gamma\}})^{(r)}(\gamma) \geq 0 \textrm{ for all } r \in \N.
\end{equation*}
The left-hand side can (essentially) be expanded using the Bonferroni inequalities to give
\begin{equation*}
\wh{1_A}^{(r)}(\gamma) -r\epsilon |A| \wh{1_A}^{(r-1)}\ast 1_{\{\gamma\}}(\gamma)\geq 0.
\end{equation*}
A short manipulation and (\ref{eqn.encodeconvolution}) then gives us that
\begin{equation*}
\epsilon |A| - r\epsilon |A|\alpha \geq 0,
\end{equation*}
which is a contradiction if $r$ is large enough in terms of $\alpha$. We now make this precise.

\begin{lemma}\label{lem.heart}
Suppose that $G$ is a finite abelian group and $A \subset G$ is symmetric of density $\alpha>0$. If there is some character $\gamma$ with $0 \leq \wh{1_A}(\gamma) \leq |A|/32$, then
\begin{equation*}
M_G(A) \geq \alpha \wh{1_A}(\gamma)/16\|1_A\|_{A(G)}^{2\alpha^{-1}+1}.
\end{equation*}
\end{lemma}
\begin{proof}
It will be useful to write $f$ for the function defined by $f(\gamma):=\max\{\wh{1_A}(\gamma),0\}$. We claim that for any positive integer $r$ we have
\begin{equation}\label{eqn.positive}
\|f^{(r)}-\wh{1_A}^{(r)}\|_{L^\infty(\wh{G})} \leq rM_G(A)\|1_A\|_{A(G)}^{r-1}.
\end{equation}
\begin{proof}[Proof of claim]
For $r=1$ this is immediate from the definition of $f$. Now, suppose we know (\ref{eqn.positive}) for some positive integer $r$. Then
\begin{eqnarray*}
\|f^{(r+1)} - \wh{1_A}^{(r+1)}\|_{L^\infty(\wh{G})} & = & \|f \ast f^{(r)}- \wh{1_A} \ast \wh{1_A}^{(r)}\|_{L^\infty(\wh{G})}\\ & \leq & \| f \ast f^{(r)} - \wh{1_A} \ast f^{(r)}\|_{L^\infty(\wh{G})}\\ & & + \|\wh{1_A} \ast f^{(r)} - \wh{1_A} \ast \wh{1_A}^{(r)}\|_{L^\infty(\wh{G})}\\ & \leq & M_G(A)\|f\|_{L^1(\wh{G})}^{r} + \|\wh{1_A}\|_{L^1(\wh{G})}\|f^{(r)} - \wh{1_A}^{(r)}\|_{L^\infty(\wh{G})},
\end{eqnarray*}
by Young's inequality and the linearity of convolution. Now, $\|\wh{1_A}\|_{L^1(\wh{G})} = \|1_A\|_{A(G)}$ and $\|f\|_{L^1(\wh{G})} \leq \|1_A\|_{A(G)}$, whence
\begin{equation*}
\|f^{(r+1)} - \wh{1_A}^{(r+1)}\|_{L^\infty(\wh{G})} \leq (r+1)M_G(A)\|1_A\|_{A(G)}^{r},
\end{equation*}
by our supposition; (\ref{eqn.positive}) now follows by induction.
\end{proof}

Let $r=2\lceil \alpha^{-1} \rceil $; the reasons for this choice of $r$ will become apparent later. Since $f \geq 0$, we have that $(f-f(\gamma)1_{\{\gamma\}})^{(r)} \geq 0$, which can be expanded using the binomial theorem to give
\begin{eqnarray}
\nonumber 0 \leq (f - f(\gamma)1_{\{\gamma\}})^{(r)}(\gamma) &= &\sum_{k=0}^{r}{(-f(\gamma))^k\binom{r}{k}f^{(r-k)} \ast 1_{\{\gamma\}}^{(k)}(\gamma)}\\ \nonumber & = &  \sum_{k=0}^{r}{(-f(\gamma)|G|^{-1})^k\binom{r}{k}f^{(r-k)}(-(k-1)\gamma)}\\ \label{eqn.keybd} & \leq & f^{(r)}(\gamma) - rf(\gamma)|G|^{-1}f^{(r-1)}(0_{\wh{G}})\\ \nonumber & & + \sum_{k=2}^{r}{(f(\gamma)|G|^{-1})^k\binom{r}{k}|f^{(r-k)}(-(k-1)\gamma)|}.
\end{eqnarray}
Now, suppose that $l \leq r$. Then by  (\ref{eqn.encodeconvolution}) and (\ref{eqn.positive}) we have that
\begin{equation*}
|f^{(l)}(\gamma)| \leq |\wh{1_A}^{(l)}(\gamma)| + lM_G(A)\|1_A\|_{A(G)}^{l-1} =  |\wh{1_A}(\gamma)| + lM_G(A)\|1_A\|_{A(G)}^{l-1} .
\end{equation*}
Trivially $|\wh{1_A}(\gamma)| \leq |A|$, and so 
\begin{equation*}
|f^{(l)}(\gamma)| \leq  |A| + rM_G(A)\|1_A\|_{A(G)}^{r-1} .
\end{equation*}
It follows that the sum in (\ref{eqn.keybd}) is at most
\begin{equation*}
\sum_{k=2}^{r}{(f(\gamma)|G|^{-1})^k\binom{r}{k}}(|A| + rM_G(A)\|1_A\|_{A(G)}^{r-1}).
\end{equation*}
Now, since $f(\gamma) \leq |A|/32 \leq |G|/r$ we get that
\begin{equation*}
\sum_{k=2}^{r}{(f(\gamma)|G|^{-1})^k\binom{r}{k}} \leq (f(\gamma)|G|^{-1})^2r^2,
\end{equation*}
and so
\begin{eqnarray*}
0 & \leq & f^{(r)}(\gamma) - rf(\gamma)|G|^{-1}f^{(r-1)}(0_{\wh{G}})\\ & & + (f(\gamma)|G|^{-1})^2r^2(|A| + rM_G(A)\|1_A\|_{A(G)}^{r-1}).
\end{eqnarray*}
Finally, equations (\ref{eqn.encodeconvolution}) and (\ref{eqn.positive}) together imply that
\begin{equation*}
f^{(r)}(\gamma) \leq f(\gamma) + rM_G(A)\|1_A\|_{A(G)}^{r-1}
\end{equation*}
and
\begin{eqnarray*}
f^{(r-1)}(0_{\wh{G}}) & \geq & \wh{1_A}^{(r-1)}(0_{\wh{G}}) - (r-1)M_G(A)\|1_A\|_{A(G)}^{r-2}\\ & = & |A| -(r-1)M_G(A)\|1_A\|_{A(G)}^{r-2}.
\end{eqnarray*}

Define $\epsilon$ by $f(\gamma)=\epsilon |A|$ and combining the above with the fact that $2\leq r\alpha \leq 4$ we have
\begin{eqnarray*}
0  & \leq  & \epsilon |A| + rM_G(A)\|1_A\|_{A(G)}^{r-1} - 2\epsilon |A| + 2(r-1)\epsilon M_G(A)\|1_A\|_{A(G)}^{r-2}\\ && + 16\epsilon^2(|A| + rM_G(A)\|1_A\|_{A(G)}^{r-1}).
\end{eqnarray*}
Rearranging all this and using the fact that $\|1_A\|_{A(G)} \geq 1$ and $\epsilon \leq 1/32$ we get that
\begin{equation*}
\epsilon |A| - 16|A|\epsilon^2 \leq 2rM_G(A)\|1_A\|_{A(G)}^{r-1}.
\end{equation*}
Thus, again since $\epsilon \leq 1/32$, we have
\begin{equation*}
M_G(A) \geq \epsilon \alpha |A| /16\|1_A\|_{A(G)}^{2\alpha^{-1}+1}.
\end{equation*}
\end{proof}

\section{Proof of Theorem \ref{thm.maintheorem}}\label{sec.maintheorem}

Theorem \ref{thm.maintheorem} follows immediately from the following more explicit version.
\begin{theorem}\label{thm.explicitdensitytheorem}
Suppose that $G$ is a finite abelian group, $A$ is a symmetric neighborhood of $0_G$ (meaning that $A$ is a symmetric set containing $0_G$) of size $\alpha |G|$ and $K \in (0,|A|/2^{13}]$ is a parameter. Then at least one of the following holds:
\begin{enumerate}
\item there is a subgroup $H \leq G$ such that $\|1_A-1_H\|_{\ell^1(G)} \leq K$;
\item there is a character $\gamma \in \wh{G}$ such that $\wh{1_A}(\gamma) \leq - K^{\alpha/5}/2^6$.
\end{enumerate}
\end{theorem}
The work of the previous section combines easily to show that either the large spectrum forms a subgroup $V$ or else there is a large negative Fourier coefficient.  In the former case $1_A$ is well approximated by $1_A \ast \mu_{V^\perp}$.  It is then easy by the trivial estimates of \S\ref{sec.trivial} that either $\wh{1_A}$ takes a large negative value or $A$ is approximately a subgroup.  The details now follow.

Before beginning the proof we require one final technical calculation which is implicit in the paper \cite{JBIS} of Bourgain.
\begin{lemma}\label{lem.calculationlemma}
Suppose that $G$ is an abelian group, $V \leq G$ and $A \subset G$. Then
\begin{equation*}
\|1_A - 1_A \ast \mu_V\|_{\ell^1(G)} = 2 \langle 1_A, 1_A -1_A \ast \mu_V\rangle,
\end{equation*}
where $\langle \cdot, \cdot \rangle$ denotes the standard inner product
\end{lemma}
\begin{proof}
$\mu_{V}$ is a probability measure so $0 \leq 1_A \ast
\mu_{V}(x) \leq 1$, hence $(1_A-1_A \ast \mu_V)(x)\leq 0$ for all $x
\not \in A$ and $(1_A-1_A \ast \mu_V)(x)\geq 0$ for all $x \in A$.
Consequently
\begin{eqnarray*}
\|1_A - 1_A \ast \mu_V\|_{\ell^1(G)}& = & \sum_{x \in G}{1_A(x)(1 - 1_A \ast \mu_V)(x)} \\ & & + \sum_{x \in G}{(1-1_A(x))1_A\ast \mu_V (x)}.
\end{eqnarray*}
However, $\sum{1_A \ast \mu_V}=\sum{1_A}$ from which the lemma follows.
\end{proof}
\begin{proof}[Proof of Theorem \ref{thm.explicitdensitytheorem}]
We may certainly suppose that $A$ is not a subgroup of $G$, or else we are trivially done with $H=A$, and so it follows from Proposition \ref{prop.vsimple} that $M_G(A) \geq 1/2$. Thus
\begin{eqnarray}\label{eqn.kyr}
M_G(A) &\geq  &(2M_G(A)+1_A(0_G))/4\\ \nonumber & = & (2M_G(A) + \int{\wh{1_A}(\gamma)d\gamma})/4 \geq \|1_A\|_{A(G)}/4,
\end{eqnarray}
where the passage from the first to the second line is via the Fourier inversion formula.

For convenience we put $\nu:=\alpha K/4\|1_A\|_{A(G)}|A|$; the reason for this choice of parameter will become clear. If there is a character $\gamma$ with $\nu |A| \leq \wh{1_A}(\gamma) \leq |A|/2^5$, then by Lemma \ref{lem.heart} we have
\begin{equation*}
M_G(A) \geq \alpha^2K/2^6\|1_A\|_{A(G)}^{2\alpha^{-1}+2},
\end{equation*}
and it follows from (\ref{eqn.kyr}) that we are in the second case of the theorem. We may thus suppose that there is no such character. Apply Lemma \ref{lem.technicallemma} with parameter $\epsilon=2^{-5}$. Either $\Spec_{\epsilon}(A)$ is a subgroup $V$ of $\wh{G}$, or
\begin{equation*}
M_G(A) \geq \alpha |A|/2^{11}\|1_A\|_{A(G)},
\end{equation*}
in which case it follows from (\ref{eqn.kyr}) that we are in the second case of the theorem, or there is a character $\gamma$ with
\begin{equation*}
|A|/2^5 =\epsilon|A|> \wh{1_A}(\gamma) \geq \epsilon^2\alpha |A|/2 \geq \nu |A|,
\end{equation*}
which contradicts our earlier supposition. Thus we may assume that $\Spec_{\epsilon}(A)$ is a subgroup $V$ of $\wh{G}$ and $\Spec_{\nu}(A) \setminus \Spec_{\epsilon}(A)=\emptyset$, so by nesting of the spectrum
\begin{equation*}
\Spec_{\nu}(A)=\Spec_{\epsilon}(A)=V \leq \wh{G}.
\end{equation*}
In view of this
\begin{equation*}
 \int_{\gamma \not \in V}{|\wh{1_A}(\gamma)|^2d\gamma} = \int_{\gamma \not \in \Spec_{\nu}(A)}{|\wh{1_A}(\gamma)|^2d\gamma}\leq \nu |A|\|1_A\|_{A(G)}.
\end{equation*}
By Parseval's theorem we have
\begin{eqnarray*}
\|1_A \ast \mu_{V^\perp}\|_{\ell^2(G)}^2 & = & \int_{\gamma \in V}{|\wh{1_A}(\gamma)|^2d\gamma} \\ & \geq & \int{|\wh{1_A}(\gamma)|^2d\gamma} - \nu|A|\|1_A\|_{A(G)} > |A|/2,
\end{eqnarray*}
whence
\begin{equation*}
\sup_{x \in G}{1_A \ast \mu_{V^\perp}(x)}.|A| \geq \|1_A \ast \mu_{V^\perp}\|_{\ell^2(G)}^2  > |A|/2,
\end{equation*}
and we conclude that there is some $x_0 \in G$ such that $1_A \ast \mu_V(x_0) > 1/2$.

Similarly, by Plancherel's theorem we have that
\begin{equation*}
\langle1_A,1_A - 1_A \ast \mu_{V^\perp}\rangle = \int_{\gamma \not \in V}{|\wh{1_A}(\gamma)|^2d\gamma} \leq \nu |A|\|1_A\|_{A(G)}.
\end{equation*}
Then, by Lemma \ref{lem.calculationlemma} we have
\begin{equation*}
\|1_A - 1_A \ast \mu_{V^\perp}\|_{\ell^1(G)} =2\langle1_A,1_A - 1_A \ast \mu_{V^\perp}\rangle \leq 2\nu |A|\|1_A\|_{A(G)}.
\end{equation*}
It follows that $1_A \ast \mu_{V^\perp}$ is $(2\nu\|1_A\|_{A(G)},1)$-almost boolean. We now apply Lemma \ref{lem.trivialdensitytheorem} to $1_A \ast \mu_{V^\perp}$ and the group $V^\perp$ to conclude that either
\begin{enumerate}
\item there is a subgroup $H \leq G$ such that
\begin{equation*}
\|1_A \ast \mu_{V^\perp} - 1_{H}\|_{\ell^1(G)} \leq 2\nu|A|\|1_A\|_{A(G)};
\end{equation*}
\item or there is a character $\gamma \in \wh{G}$ such that
\begin{equation*}
(1_A \ast \mu_{V^\perp})^\wedge(\gamma) \leq -|V^\perp|/8 +2\nu|A||1_A\|_{A(G)}.
\end{equation*}
\end{enumerate}
In the first instance, by the triangle inequality we have
\begin{eqnarray*}
\|1_A - 1_H\|_{\ell^1(G)} & \leq & \|1_A - 1_A \ast \mu_{V^\perp}\|_{\ell^1(G)} + \|1_A \ast \mu_{V^\perp}-1_{H}\|_{\ell^1(G)}\\ & \leq & 4\nu |A|\|1_A\|_{A(G)}\leq K,
\end{eqnarray*}
and we find ourselves in the first case of the theorem.

In the second instance since $(1_A \ast \mu_{V^\perp})^{\wedge}(\gamma) = \wh{1_A}(\gamma)1_{V}(\gamma)$ we see that
\begin{equation}\label{eqn.lwe}
M_G(A) \geq |V^\perp|/8 -2\nu|A|\|1_A\|_{A(G)}.
\end{equation}
By Parseval's theorem we have
\begin{equation*}
|A| = \int{|\wh{1_A}(\gamma)|^2d\gamma} \geq \mu(V).\epsilon^2|A|^2 = \mu(V)|A|^2/2^{10},
\end{equation*}
and it follows that $\mu(V) \leq 2^{10}/|A|$, and hence that $|V^\perp| \geq |A|/2^{10}$. Finally, inserting this bound into (\ref{eqn.lwe}) places us in the second case of the theorem and we are done -- in fact in this case we have $M_G(A)=\Omega(|A|)$, however earlier parts of the proof led to the weaker conclusion in the second case.
\end{proof}

\section{Bourgain systems}\label{sec.bs}

In this section we recall the notion of Bourgain system from the paper \cite{BJGTS2}. Although formally new in that paper the material of this section is morally standard c.f. \cite{TCTVHV}.

We should remark that in \cite{BJGTS2} all the results are stated for finite abelian groups. There is no change in the passage to finite systems in discrete abelian groups which is the case we shall need; we shall make no further comment on the matter.

Suppose that $G$ is an abelian group and $d \geq 1$ is an integer. A
\emph{Bourgain system} $\mathcal{B}$ of dimension $d$ is a
collection $(B_\rho)_{\rho \in (0,2]}$ of finite subsets of $G$ such that
the following axioms are satisfied:
\begin{enumerate}
\item (Nesting) If $\rho' \leq \rho$ we have $B_{\rho'} \subseteq B_{\rho}$;
\item (Zero) $0_G \in B_\rho$ for all $\rho \in (0,2]$;
\item (Symmetry) If $x \in B_{\rho}$ then $-x \in B_{\rho}$;
\item (Addition) For all $\rho,\rho'$ such that $\rho + \rho' \leq 1$ we have $B_{\rho} + B_{\rho'} \subseteq B_{\rho + \rho'}$;
\item (Doubling) If $\rho \leq 1$ then $|B_{2\rho}| \leq 2^d|B_\rho|$.
\end{enumerate}
We define the $\emph{size}$ of $\mathcal{B}=(B_\rho)_\rho$ to be
$|B_1|$ and denote it $|\mathcal{B}|$. Frequently we shall
consider several Bourgain systems
$\mathcal{B},\mathcal{B}',\mathcal{B}'',...$; in this case the
underlying sets will be denoted
$(B_\rho)_\rho,(B'_\rho)_\rho,(B''_\rho)_\rho,...$. We say that a Bourgain system $\mathcal{B}$ is a \emph{sub-system} of $\mathcal{B}'$ if $B_\rho \subset B'_\rho$ for all $\rho$. 

It may be useful to keep some examples of Bourgain systems in mind: the prototypes are coset progressions first introduced by Green and Ruzsa in \cite{BJGIZR} in their proof of Fre{\u\i}man's theorem in general abelian groups.

\begin{example}[Coset progressions]
Suppose that $G$ is an abelian group, $H \leq G$ is finite, $x \in G^d$ and $L \in \N^d$. We define the \emph{coset progression} $\Prog(H,x,L)$ to be the set
\begin{equation*}
\{h+l_1.x_1+\dots + l_d.x_d: h \in H,|l_i| \leq L_i \textrm{ for all } i \in \{1,\dots,d\}\}.
\end{equation*}
The system $\mathcal{B}:=(\Prog(H,x,\rho L))_{\rho \in (0,2]}$ is easily seen to be a Bourgain system and since
\begin{equation*}
\Prog(H,x,2\rho L) \subset \Prog(\{0_G\},x',1) + \Prog(H,x,\rho L)
\end{equation*}
where $x'=(\lceil \rho L_1\rceil .x_1,\dots,\lceil \rho L_d\rceil  .x_d)$ and $1=(1,\dots,1)$, we see that it is $O(d)$-dimensional.
\end{example}

In a qualitative sense Fre{\u\i}man's theorem shows that all Bourgain systems are essentially coset progressions.  Indeed, suppose that $\mathcal{B}$ is an $O(1)$-dimensional Bourgain system then $|B_1+B_1| =O(|B_1|)$ and so by Fre{\u\i}man's theorem there is an $O(1)$-dimensional coset progression $\Prog(H,x,L)$ of size $O(|B_1|)$ which contains $B_1$.  Quantitatively it is worth being more subtle and dealing with the more abstract Bourgain system.

For more examples and a detailed explanation the reader may wish to consult \cite{BJGTS2}.

The following trivial lemma gives us a useful bound for the size of a low-dimensional Bourgain system.\begin{lemma}[{\cite[Lemma 4.4]{BJGTS2}}]\label{lem.bourgainsize}
Suppose that $G$ is an abelian group, $\mathcal{B}$ is a Bourgain system of dimension $d$ and
$\lambda \in (0,1]$ is a parameter. Then $\lambda
\mathcal{B}:=(B_{\lambda\rho})_{\rho}$ is a Bourgain system of
dimension $d$ and size at least
$(\lambda/2)^d|\mathcal{B}|$.
\end{lemma}
Not all Bourgain systems behave as well as we would like; we
say that a Bourgain system $\mathcal{B}$ of dimension $d$ is
\emph{regular} if
\begin{equation*}
1-10d|\eta| \leq \frac{|B_1|}{|B_{1+\eta}|}\leq 1+
10d|\eta|
\end{equation*}
for all $\eta$ with $d|\eta| \leq 1/10$. Typically, however,
Bourgain systems are regular, a fact implicit in the usual proof of the
following proposition.
\begin{proposition}[{\cite[Lemma 4.12]{BJGTS2}}]\label{prop.ubreg}
Suppose that $G$ is an abelian group and $\mathcal{B}$ is a Bourgain system of dimension $d$.
Then there is a $\lambda \in [1/2,1)$ such that $\lambda\mathcal{B}$
is regular.
\end{proposition}
We associate to $\mathcal{B}$ a system of measures denoted $(\beta_\rho)_\rho$ defined by $\beta_\rho=\mu_{B_\rho} \ast \mu_{B_\rho}$ where $\mu_{B_\rho}$ denotes the uniform probability measure with support $B_\rho$. It is more natural to take the measures $\mu_{B_\rho}$ rather than $\beta_\rho$, however certain positivity requirements in \cite{BJGTS2} precipitated the use of these convolved measures and we shall in fact further leverage this convenience in the proof of Corollary \ref{cor.kyt} below.
\begin{lemma}[{\cite[Lemma 4.13]{BJGTS2}}]\label{lem.haar}
Suppose that $G$ is an abelian group, $\mathcal{B}$ is a regular Bourgain system of dimension
$d$ and $y \in B_{\eta}$. Then 
\begin{equation*}
\|(y+\beta) - \beta\| \leq 20d\eta,
\end{equation*}
where we recall that $y+\beta$ denotes the measure induced by $f \mapsto \int{f(x)d\beta(y+x)}$.
\end{lemma}
\begin{lemma}[{\cite[Lemma 4.15]{BJGTS2}}]\label{lem.cont}
Suppose that $G$ is an abelian group, $\mathcal{B}$ is a regular Bourgain system of dimension
$d$ and $f \in \ell^\infty(G)$. Then
\begin{equation*}
\sup_{x \in G}{\|f \ast \beta - f \ast \beta(x)\|_{\ell^\infty(x+B_\eta)}} \leq
20\|f\|_{\ell^\infty(G)}d\eta.
\end{equation*}
\end{lemma}
The previous lemma encodes the idea that $f\ast \beta$ is in some sense continuous. We shall make use of this by way of a sort of intermediate value theorem; this sort of idea appeared first in \cite{BJGSVK}.
\begin{lemma}\label{lem.ivt}
Suppose that $G$ is an abelian group, $\mathcal{B}$ is a regular Bourgain system of dimension $d$ and $f:G \rightarrow [-1,1]$ is such that $f\ast \beta$ is $(\epsilon,\infty)$-almost boolean for some $\epsilon \in (0,1/3)$. Then $S:=\{x : f \ast \beta(x)>1/2\}$ is constant on cosets of $V$, the group generated by $B_{\epsilon/20d}$.
\end{lemma}
\begin{proof}
Suppose that $x \in S$ so 
\begin{equation*}
|f \ast \beta(x) -1| \leq \epsilon \|f \ast \beta\|_{\ell^\infty(G)} \leq \epsilon \|f\|_{\ell^\infty(G)}\|\beta\| \leq \epsilon.
\end{equation*}
Now, if $y \in B_{\epsilon/20d}$ then
\begin{equation*}
|f \ast \beta(x+y) - f \ast \beta(x)| \leq \epsilon\|f\|_{\ell^\infty(G)}/2 \leq \epsilon
\end{equation*}
by Lemma \ref{lem.cont}. Furthermore, $f \ast \beta$ is $(\epsilon,\infty)$-almost boolean whence there is some $z \in \{0,1\}$ such that
\begin{equation*}
|z-f \ast \beta(x+y)| \leq \epsilon\|f\|_{\ell^\infty(G)} \leq \epsilon.
\end{equation*}
Combining these three expressions using the triangle inequality we get that
\begin{equation*}
|z-1| \leq |z-f \ast \beta(x+y)| + |f \ast \beta(x+y) - f \ast \beta(x)| + |f \ast \beta(x) -1|  \leq 3\epsilon <1.
\end{equation*} 
It follows that $z=1$ so that $x+y \in S$. Thus $S = S + B_{\epsilon/20d}$ and we arrive at the result.
\end{proof}

\section{Quantitative notions of continuity in $A(G)$}\label{sec.tools}

A key tool in the paper \cite{BJGTS2} was a localization of an argument of Green and Konyagin \cite{BJGSVK} to Bourgain systems. Roughly their result gave a quantitative interpretation of the qualitative fact that if $f \in A(G)$ then $f$ is (essentially) continuous. Specifically we require the following proposition which can be read out of the proof of \cite[Proposition 5.1]{BJGTS2}.
\begin{proposition}\label{prop.quantcont}
Suppose that $G$ is an abelian group, $A$ is a finite subset of $G$, $\mathcal{B}$ is a regular Bourgain system of dimension $d$ and $\epsilon \in (0,1]$ is a parameter. Then there is a regular sub-system $\mathcal{B}'$ with
\begin{equation*}
\dim \mathcal{B}' =O(d + \epsilon^{-2}\|1_A\|_{A(G)}^2)
\end{equation*}
and
\begin{equation*}
|\mathcal{B}'| \geq \exp(-O(\epsilon^{-4}\|1_A\|_{A(G)}^4d (1+\log \epsilon^{-1} \|1_A\|_{A(G)}d)))|\mathcal{B}|,
\end{equation*}
such that
\begin{equation*}
\sup_{x \in G}{\|1_A- 1_A \ast \beta'\|_{L^2(x+B_\rho')}} \leq \epsilon
\end{equation*}
for every $\rho \in [\epsilon/160\dim \mathcal{B}',\epsilon/80\dim \mathcal{B}]$ for which $\rho\mathcal{B}'$ is regular, where we recall that $x+B_\rho'$ is endowed with $x+\beta_\rho'$, the measure induced by $f \mapsto \int{f(y)d\beta_\rho'(x+y)}$.
\end{proposition}
It should be remarked that it is possible to improve the powers of $\|1_A\|_{A(G)}$ and $\epsilon$ in this theorem and doing so results in an improvement to the power of $\log$ in Theorem \ref{thm.densityindependentmaintheorem}.

We also require the celebrated Balog-Szemer{\'e}di and Fre{\u\i}man theorems of \cite{ABES} and \cite{GAF}; see \cite{TCTVHV} for a comprehensive discussion. Our use follows the time honored method laid down by Gowers in \cite{WTG} and the weakness of the powers in Proposition \ref{prop.weakfrei} is the main reason we have not given an explicit constant for the power of $\log$ in Theorem \ref{thm.densityindependentmaintheorem}. There is some hope that this may be remedied if the arguments of \cite{BJGTCTF} are transfered to the general setting.

The following can be read out of the proof of \cite[Proposition 6.3]{BJGTS2}.
\begin{proposition}\label{prop.weakfrei} 
Suppose that $G$ is an abelian group and $A$ is a finite subset of $G$ with $\|1_A \ast 1_A \|_{\ell^2(G)}^2 \geq c |A|^3$. Then there is a regular Bourgain system $\mathcal{B}$ with
\begin{equation*}
\dim \mathcal{B} = c^{-O(1)}\textrm{ and } |\mathcal{B}| = \exp(c^{-O(1)})|A|
\end{equation*}
such that
\begin{equation*}
\|1_A \ast 1_A \ast \beta\|_{\ell^2(G)}^2 =c^{-O(1)}|A|^3.
\end{equation*}
\end{proposition}
The key result of this section is the following which will be the only result that we require again from this or the previous section.
\begin{corollary}\label{cor.kyt}
Suppose that $G$ is an abelian group and $A' \subset A$ are non-empty, finite subsets of $G$ with $\|1_{A'} \ast 1_{A'}\|_{\ell^2(G)}^2 \geq c|A'|^3$ and $\epsilon \in (0,1/2)$ is a parameter. Then there is a subgroup $V \leq G$ with
\begin{equation*}
|V| =\exp(-\epsilon^{-O(1)}c^{-O(1)}\|1_A\|_{A(G)}^{O(1)})|A'|
\end{equation*}
such that $1_A \ast \mu_V$ is $(\epsilon,\infty)$-almost boolean and an $a' \in A'$ such that $1_A \ast \mu_V(a')>1/2$.
\end{corollary}
\begin{proof}
Apply Proposition \ref{prop.weakfrei} to the set $A'$ to get a regular Bourgain system $\mathcal{B}$ with\begin{equation*}
\dim \mathcal{B} = c^{-O(1)} \textrm{ and } |\mathcal{B}| =\exp(-c^{-O(1)})|A'| 
\end{equation*}
such that
\begin{equation*}
\|1_{A'} \ast 1_{A'} \ast \beta\|_{\ell^2(G)}^2 =c^{-O(1)}|A'|^3.
\end{equation*}
Set the parameter $\eta:=|A'|^{-3}\|1_{A'} \ast 1_{A'}\ast \beta\|_{\ell^2(G)}^2\epsilon/12$ and apply Proposition \ref{prop.ubreg} to pick a $\lambda$ with
\begin{equation}\label{eqn.tip}
\eta/20\dim \mathcal{B} \geq \lambda \geq \eta/40\dim \mathcal{B}
\end{equation}
such that $\lambda \mathcal{B}$ is regular. 

Now, apply Proposition \ref{prop.quantcont} to $A$ (\emph{not} $A'$) with the regular Bourgain system $\lambda\mathcal{B}$ and parameter $\eta$ to get a regular sub-system $\mathcal{B}'$ with
\begin{equation*}
\dim \mathcal{B}' = \epsilon^{-O(1)} c^{-O(1)}\|1_A\|_{A(G)}^{O(1)}
\end{equation*}
and
\begin{equation*}
|\mathcal{B}'|  \geq\exp(-\epsilon^{-O(1)} c^{-O(1)}\|1_A\|_{A(G)}^{O(1)})|A'| 
\end{equation*}
such that
\begin{equation}\label{eqn.y}
\sup_{x \in G}{\|1_A - 1_A \ast \beta'\|_{L^2(x+B_\rho')}} \leq \eta
\end{equation}
for all $\rho \in [\eta/160\dim \mathcal{B}',\eta/80\dim \mathcal{B}]$ such that $\rho\mathcal{B}'$ is regular.

Given this, apply Lemma \ref{lem.cont} and Proposition \ref{prop.ubreg} to get a regular $\rho$ with $\rho\in [\eta/160\dim \mathcal{B}',\eta/80\dim \mathcal{B}]$ such that
\begin{equation}\label{eqn.r}
\sup_{x \in G}{\|1_A \ast \beta' - 1_A \ast \beta'(x)\|_{L^\infty(x+B_\rho')}} \leq \eta.
\end{equation}
Let $V$ be the group generated by $B_\rho'$. Then, by Lemma \ref{lem.bourgainsize} we have that
\begin{equation*}
|V| \geq |B_\rho'| \geq (\rho/2)^{\dim \mathcal{B}'}|\mathcal{B}'| \geq \exp(-\epsilon^{-O(1)}c^{-O(1)}\|1_A\|_{A(G)}^{O(1)})|A'|,
\end{equation*}
as desired.

Suppose that there is some $x_0 \in G$ such that $|1_A \ast \beta'(x_0)-z| > 2\eta$ for all $z\in \{0,1\}$. Then we see that $|1_A \ast \beta'(x)-z| > \eta$ for all $x \in x_0+B_{\rho}'$ by (\ref{eqn.r}). However, integrating this contradicts (\ref{eqn.y}). It follows that $1_A \ast \beta'$ is $(4\eta,\infty)$-almost boolean.

Writing $S:=\{x \in G: 1_A \ast \beta'(x)>1/2\}$ we see from the definition of $V$ and Lemma \ref{lem.ivt} that $S$ is constant on cosets of $V$. Now
\begin{eqnarray*}
|1_A \ast \mu_V(x) - 1_{S} \ast \mu_V(x)| & \leq & |1_A-1_{S}| \ast
\mu_V(x)\\ & = & |1_A -1_{S}| \ast \beta_\rho' \ast \mu_V(x)\\ &
\leq & \left(|1_A -1_{S}|^2 \ast \beta_\rho'\right)^{\frac{1}{2}}
\ast \mu_V(x)
\end{eqnarray*}
by the Cauchy-Schwarz inequality. Hence
\begin{eqnarray*}
|1_A \ast \mu_V(x) - 1_{S}\ast \mu_V(x)| & \leq & \left(|1_A -1_A\ast
\beta'|^2 \ast \beta_\rho'\right)^{\frac{1}{2}} \ast \mu_V(x)\\
& & + \left(|1_A\ast \beta' -1_{S}|^2 \ast
\beta_\rho'\right)^{\frac{1}{2}} \ast \mu_V(x)\\ & \leq &5\eta.
\end{eqnarray*}
Since $S$ is constant on cosets of $V$ we have that $1_S=1_S \ast \mu_V$ and hence conclude that $1_A \ast \mu_V$ is $(10\eta,\infty)$-almost boolean as required (in view of the definition of $\eta$).

On the other hand by Lemma \ref{lem.haar}, the upper bound on $\lambda$ and the fact that $\mathcal{B}'$ is a sub-system of $\lambda \mathcal{B}$ we get that
\begin{equation*}
\|\beta - \beta \ast \beta'\| \leq \eta.
\end{equation*}
It follows that
\begin{equation*}
\|1_{A'} \ast 1_{A'} \ast \beta \ast \beta'\|_{\ell^2(G)}^2 \geq \|1_{A'} \ast 1_{A'} \ast \beta\|_{\ell^2(G)}^2 -\eta|A'|^3 \geq 11\eta |A'|^3.
\end{equation*}
But, by Parseval's theorem
\begin{eqnarray*}
\|1_{A'} \ast 1_{A'} \ast \beta \ast \beta'\|_{\ell^2(G)}^2 & = & \int{|\wh{1_{A'}}(\gamma)\wh{\beta}(\gamma)\wh{\beta'}(\gamma)|^2d\gamma}\\ & \leq & |A'|^2\int{|\wh{1_{A'}}(\gamma)|^2\wh{\beta'}(\gamma)d\gamma}
\end{eqnarray*}
since $\wh{\beta'}(\gamma)\geq 0$. Combining these with Plancherel's theorem tells us that
\begin{equation*}
\langle 1_{A'} \ast \beta',1_{A'}\rangle \geq11\eta|A'|.
\end{equation*}
It follows from H{\"o}lder's inequality that there is some $a' \in A'$ for which $1_{A'} \ast \beta'(a') \geq11\eta$.

It remains to note that $A' \subset A$ whence $1_A \ast \beta'(a') \geq 1_{A'} \ast \beta'(a')$. Since $1_A \ast \beta'$ is $(10\eta,\infty)$-almost boolean we see that $1_A \ast \beta'(a') \geq 1-\epsilon$; it follows that $a' \in S$ and so $1_A \ast \mu_V(a') \geq 1-2\epsilon > 1/2$. The proof is complete.
\end{proof}

\section{Proof of Theorem \ref{thm.densityindependentmaintheorem}}\label{sec.maintheorem2}

Our proof of Theorem \ref{thm.densityindependentmaintheorem} is iterative in nature with the next lemma being driver. We briefly sketch the statement in words to aid understanding.

We build up a collection of subgroups. At each stage if $M_G(A)$ is small and $A$ is not essentially the union of the subsgroups we have already found then we may find another subgroup which is `very orthogonal' to those already found and which is almost entirely contained in $A$.

The orthogonality coupled with the algebra norm bound (resulting from the fact that $M_G(A)$ is small) implies that the iteration cannot proceed for too many steps.
\begin{lemma}\label{lem.mainitlem}
Suppose that $G$ is an abelian group, $A$ is a non-empty, finite symmetric subset of $G$ and $2^{-4}K'\geq K \geq 1$ are parameters. Suppose, further, that $\mathcal{H}$ is a finite collection of subgroups of $G$ with
\begin{equation*}
|H| \geq K', |H \setminus A| \leq K \textrm{ and } \sup_{x \not \in H}{1_A \ast \mu_H(x)} \leq 1/16\|1_A\|_{A(G)} \textrm{ for all } H \in \mathcal{H}
\end{equation*}
such that
\begin{equation*}
|H \cap H'| \leq K \textrm{ for all distinct pairs }H,H' \in \mathcal{H}.
\end{equation*}
Then, recalling that $\bigcup{\mathcal{H}}=\bigcup_{H \in \mathcal{H}}{H}$, at least one of the following is true:
\begin{enumerate}
\item \label{enum.it.c1} (Good approximation) 
\begin{equation*}
|A \setminus \bigcup{\mathcal{H}}| \leq 2^4|\mathcal{H}|^2K;
\end{equation*}
\item \label{enum.it.c2} (Large negative Fourier coefficient)
\begin{equation*}
M_G(A) = K^{\Omega(1)};
\end{equation*}
\item \label{enum.it.c3} (Unbalanced parameters)
\begin{equation*}
K' \geq \exp(-(\|1_A\|_{A(G)} + |\mathcal{H}|)^{C_{\mathcal{S}}})|A\setminus \bigcup{\mathcal{H}}|
\end{equation*}
for some absolute $C_\mathcal{S}>0$;
\item \label{enum.it.c4} (Correlating subgroup) there is a subgroup $H_0 \leq G$ with
\begin{equation*}
|H_0| \geq K', |H_0 \setminus A| \leq K \textrm{ and } \sup_{x \not \in H_0}{1_A \ast \mu_{H_0}(x)} \leq 1/16\|1_A\|_{A(G)}
\end{equation*}
such that
\begin{equation*}
|H_0 \cap H| \leq K \textrm{ for all } H \in \mathcal{H}.
\end{equation*}
\end{enumerate}
\end{lemma}
\begin{proof}
Begin by considering the function $g:=\sum_{H \in \mathcal{H}}{1_{H}}$. By the triangle inequality and the fact that $1_H^2=1_H$, we have
\begin{equation*}
\|g^2-g\|_{\ell^1(G)} = \sum_{H \neq H', H,H' \in \mathcal{H}}{|H\cap H'|} \leq |\mathcal{H}|^2K,
\end{equation*}
i.e. $g$ behaves quite a lot like a boolean function: the indicator function of the set $\bigcup{\mathcal{H}}$. In particular, $g$ is non-negative and if $x \in \supp g= \bigcup{\mathcal{H}}$, then $g(x) \geq 1$, so we have that
\begin{eqnarray*}
\|g-1_{\bigcup{\mathcal{H}}}\|_{\ell^2(G)}^2 = \sum_{x \in \bigcup{\mathcal{H}}}{(g(x)-1)^2} & \leq & \sum_{x \in G}{g(x)(g(x)-1)}\\ & = & \|g^2-g\|_{\ell^1(G)}  \leq |\mathcal{H}|^2K.
\end{eqnarray*}
Furthermore, since $g-1_{\bigcup{\mathcal{H}}}$ takes values in $\N_0$ we have that
\begin{equation}\label{eqn.gest}
\|g-1_{\bigcup{\mathcal{H}}}\|_{\ell^1(G)}  \leq \|g-1_{\bigcup{\mathcal{H}}}\|_{\ell^2(G)}^2 \leq |\mathcal{H}|^2K.
\end{equation}

Now, let $f$ be the function $1_A - g$. It follows from the triangle inequality and the fact that the algebra norm of a subspace is $1$ that
\begin{equation*}
\|f\|_{A(G)} \leq \|1_A\|_{A(G)} + \sum_{H \in \mathcal{H}}{\|1_H\|_{A(G)}} \leq \|1_A\|_{A(G)} + |\mathcal{H}|.
\end{equation*}
By Parseval's theorem we have that
\begin{equation*}
\|f \ast f\|_{\ell^2(G)}^2 = \|\wh{f}\|_{L^4(\wh{G})}^4.
\end{equation*}
However, by $\log$-convexity of the $L^p(\wh{G})$-norms and Parseval's theorem
\begin{equation*}
\|\wh{f}\|_{L^4(\wh{G})}^4 \geq \|\wh{f}\|_{L^1(\wh{G})}^{-2}\|\wh{f}\|_{L^2(\wh{G})}^6=\|f\|_{A(G)}^{-2}\|f\|_{\ell^2(G)}^6,
\end{equation*}
whence
\begin{equation}\label{eqn.nrg}
\|f \ast f\|_{\ell^2(G)}^2 \geq (\|1_A\|_{A(G)}+|\mathcal{H}|)^{-2}\|f\|_{\ell^2(G)}^6.
\end{equation}

Write $A'$ for the set $A \setminus \bigcup{\mathcal{H}}$ and $E$ for the set $\bigcup{\mathcal{H}} \setminus A$, so that
\begin{equation*}
f=1_{A'} +(1_{\bigcup{\mathcal{H}}}-g) -1_E.
\end{equation*}
Since $|H \setminus A| \leq K$ for all $H \in \mathcal{H}$ we have $|E|\leq K|\mathcal{H}|$, so by the triangle inequality for the $\ell^2(G)$-norm and (\ref{eqn.gest}) we have
\begin{eqnarray*}
\|f\|_{\ell^2(G)} & \geq &\|1_{A'}\|_{\ell^2(G)}-\|1_{\bigcup{\mathcal{H}}}-g\|_{\ell^2(G)} -\|1_E\|_{\ell^2(G)}\\ &\geq & \sqrt{|A'|} - 2|\mathcal{H}|\sqrt{K}.
\end{eqnarray*}
Thus, either $|A'| \leq 2^4|\mathcal{H}|^2K$ and we are in case (\ref{enum.it.c1}), or else
\begin{equation}\label{eqn.t}
\|f\|_{\ell^2(G)}^2 \geq |A'|/4.
\end{equation}
By the triangle inequality for the $L^4(\wh{G})$-norm we have
\begin{equation}\label{eqn.rf}
\|\wh{f}\|_{L^4(\wh{G})} \leq \|\wh{1_{A'}}\|_{L^4(\wh{G})} + \|(1_{\bigcup{\mathcal{H}}}-g)^\wedge\|_{L^4(\wh{G})}+\|\wh{1_E}\|_{L^4(\wh{G})}.
\end{equation}
However by H{\"o}lder's inequality, the Hausdorff-Young inequality and Parseval's theorem we have that
\begin{eqnarray*}
\|(1_{\bigcup{\mathcal{H}}}-g)^\wedge\|_{L^4(\wh{G})}^4&\leq& \|(1_{\bigcup{\mathcal{H}}}-g)^\wedge\|_{L^2(\wh{G})}^2\|(1_{\bigcup{\mathcal{H}}}-g)^\wedge\|_{L^\infty(\wh{G})}^2\\ & \leq &\|1_{\bigcup{\mathcal{H}}}-g\|_{\ell^2(G)}^2\|1_{\bigcup{\mathcal{H}}}-g\|_{\ell^1(G)}^2 \leq |\mathcal{H}|^6K^3.
\end{eqnarray*}
Where the last inequality is from (\ref{eqn.gest}). Similarly
\begin{equation*}
\|\wh{1_E}\|_{L^4(\wh{G})}^4 \leq \|\wh{1_E}\|_{L^2(\wh{G})}^2\|\wh{1_E}\|_{L^\infty(\wh{G})}^2\leq \|1_E\|_{\ell^2(G)}^2\|1_E\|_{\ell^1(G)}^2 \leq K^3|\mathcal{H}|^3.
\end{equation*}
Inserting these estimates into (\ref{eqn.rf}) we get that
\begin{equation*}
\|\wh{f}\|_{L^4(\wh{G})} \leq \|\wh{1_{A'}}\|_{L^4(\wh{G})} + 2(|\mathcal{H}|^6K^3)^{1/4}.
\end{equation*}
Now, by Parseval's theorem we have
\begin{equation*}
\|\wh{1_{A'}}\|_{L^4(\wh{G})}^4 = \|1_{A'} \ast 1_{A'}\|_{\ell^2(G)}^2 \geq |A'|^2.
\end{equation*}
Thus either $|A'|^8 \leq 2^4|\mathcal{H}|^{6}K^3$ and we are in case (\ref{enum.it.c1}) or else
\begin{equation*}
\|1_{A'} \ast 1_{A'}\|_{\ell^2(G)}^2 = \|\wh{1_{A'}}\|_{L^4(\wh{G})}^4 = \Omega(\|\wh{f}\|_{L^4(\wh{G})}^4) = \Omega((\|1_A\|_{A(G)} + |\mathcal{H}|)^{-2}|A'|^3)
\end{equation*}
by Parseval's theorem, (\ref{eqn.nrg}) and (\ref{eqn.t}).

Now, apply Corollary \ref{cor.kyt} to $A' \subset A$ with parameter $\epsilon=1/16\|1_A\|_{A(G)}$ to get a subgroup $V \leq G$ with
\begin{equation*}
|V|=\exp(-(\|1_A\|_{A(G)} + |\mathcal{H}|)^{O(1)})|A'| 
\end{equation*}
such that $1_A \ast \mu_V$ is $(\epsilon,\infty)$-almost boolean and an $a' \in A'$ such that $1_A \ast \mu_V(a') > 1/2$. We let
\begin{equation*}
H_0:=\{x \in G: 1_A \ast \mu_V(x)>1/2\},
\end{equation*}
and we shall now show that $H_0$ has the necessary properties to be the group in case (\ref{enum.it.c4}) of the lemma.

\begin{claim}
$H_0$ is a subgroup.
\end{claim}
\begin{proof}
Apply Lemma \ref{lem.trivialdensity2} to $1_A \ast \mu_V$. This tells us that either $H_0$ is a subgroup or there is a character $\gamma$ such that
\begin{equation*}
\wh{1_A}(\gamma)1_{V^\perp}(\gamma) = (1_A \ast \mu_V)^\wedge(\gamma) \leq -|V|/16 
\end{equation*}
since $\epsilon \leq 1/10$. In view of the lower bound on $|V|$ it follows that either we are in case (\ref{enum.it.c3}) or (we aren't and are therefore) in case (\ref{enum.it.c2}). It follows that we may assume that $H_0$ is a subgroup.
\end{proof}

\begin{claim}
$K' \leq |H_0| <\infty$
\end{claim}
\begin{proof}
Since $V \subset H_0$ we see that either we are in (\ref{enum.it.c3}) or else $|H_0| \geq K'$ as required. The upper bound follows since $A$ is finite and $1_A \ast \mu_V(x)>1/2$ for all $x \in H_0$.
\end{proof}

\begin{claim}
$|H_0 \setminus A| \leq K$
\end{claim}
\begin{proof}
Apply Theorem \ref{thm.maintheorem} to the set $A \cap H_0$ (possible since $H_0$ is finite), so that either there is a subgroup $H' \leq H_0$ such that
\begin{equation}\label{eqn.as}
\|1_{A\cap H_0} -1_{H'}\|_{\ell^1(H_0)} \leq K
\end{equation}
or else we have a character $\gamma$ (on $H_0$ which induces a character on $G$) such that
\begin{equation*}
\wh{1_A}\ast \mu_{H_0^\perp}(\gamma) =(1_{A \cap H_0})^\wedge(\gamma) \leq -K^{\Omega(1)}.
\end{equation*}
It follows that we are in (\ref{enum.it.c2}) by averaging since $\mu_{H_0^\perp}$ is a probability measure.

Since $1_A \ast \mu_V$ is $(\epsilon,\infty)$-almost boolean (and $H_0$ is a subgroup so $0_G \in H_0$) we have that
\begin{equation*}
1_A \ast \mu_V(x)>3/4 \textrm{ for all } x \in H_0.
\end{equation*}
Furthermore, $V$ is a subgroup of $H_0$, so we have that $1_A \ast \mu_{H_0}(0_G)>3/4$.

Now,  $H' \leq H_0$ and if it were a proper subgroup then we would have that $|H'| \leq |H_0|/2$ whence
\begin{equation*}
|H_0|/4 < \|1_{A\cap H_0} - 1_{H'}\|_{\ell^1(H_0)}  \leq K.
\end{equation*}
by (\ref{eqn.as}). Since $|H_0| \geq |V|$ we conclude that we are in case (\ref{enum.it.c3}). Thus we may suppose not so that $H'=H_0$ and it follows that $|H_0 \setminus A| \leq K$ as required.
\end{proof}

\begin{claim}
$|H \cap H_0| \leq K$ for all $H \in \mathcal{H}$
\end{claim}
\begin{proof}
Suppose that $H \in \mathcal{H}$. Since $H_0 \cap A' \neq \emptyset$ and $A' \cap H = \emptyset$ we see that $H_0 \not \leq H$, whence $|H_0+H| \geq 2|H|$.

Let $H_1:=H \cap H_0$ and consider the inner product
\begin{equation*}
\langle 1_A\ast \mu_{H_1},(\mu_{H_1} - \mu_{H}) \ast (\mu_{H_1} - \mu_{H_0})\rangle.
\end{equation*}
When expanded out it is equal to
\begin{equation*}
1_A \ast \mu_{H_1}(0_G) - 1_A \ast \mu_H(0_G) - 1_A \ast \mu_{H_0}(0_G) + 1_A \ast \mu_{H_0+H}(0_G).
\end{equation*}
Now the first term is at most $1$, the second and third at least $1-K/|H|$ and, finally, the fourth is at most
\begin{equation*}
\frac{1}{|H_0+H|}.\left(|A \cap H| + \frac{|H_0+H|-|H|}{|H|}\sup_{x \not \in H}{1_A \ast \mu_H(x)}\right) \leq (1+1/2)/2.
\end{equation*}
Combining all this tells us that
\begin{equation*}
\langle 1_A\ast \mu_{H_1},(\mu_{H_1} - \mu_{H}) \ast (\mu_{H_1} - \mu_{H_0})\rangle \leq -1/8.
\end{equation*}
By Plancherel's theorem we conclude that
\begin{equation*}
-1/8 \geq \int_{\gamma \in H_1^\perp}{\wh{1_A}(\gamma)(1_{H_1^\perp} - 1_{H^\perp})(\gamma)(1_{H_1^\perp} - 1_{H_0^\perp})(\gamma)d\gamma} \geq \inf_{\gamma \in \wh{G}}{\wh{1_A}(\gamma)}\mu(H_1^\perp).
\end{equation*}
Rearranging it follows that $M_G(A) = \Omega(|H_1|)$. Thus we are either in case (\ref{enum.it.c2}) or else $|H_1| \leq K$ as desired.
\end{proof}

\begin{claim}
$\sup_{x \not \in H_0}{1_A \ast \mu_{H_0}(x)} \leq 1/16\|1_A\|_{A(G)}$
\end{claim}
\begin{proof}
If $x \not \in H_0$ then $1_A \ast \mu_V(x) \leq \epsilon$ since $1_A \ast \mu_V(x)$ is $(\epsilon,\infty)$-almost boolean whence the desired conclusion follows on noting that $V$ is a subgroup of $H_0$.
\end{proof}

It follows that $H_0$ has all the claimed properties and we are in case (\ref{enum.it.c4}); the proof is complete.
\end{proof}
We are now in a position to iterate the above lemma to prove the main theorem.
\begin{proof}[Proof of Theorem \ref{thm.densityindependentmaintheorem}]
Define the auxiliary parameter $K_0$ to be
\begin{equation*}
\min\{|A \triangle \bigcup{\mathcal{H}}|: \mathcal{H} \textrm{ is a collection of at most } 32M_G(A)\textrm{ subgroups.}\}.
\end{equation*}
We begin as in the proof of Theorem \ref{thm.maintheorem} and may suppose that $A$ is not a subgroup of $G$, or else we are trivially done with $H=A$, and so it follows from Proposition \ref{prop.vsimple} that $M_G(A) \geq 1/2$. Thus
\begin{eqnarray*}
M_G(A) &\geq  &(2M_G(A)+1_A(0_G))/4\\ \nonumber & = & (2M_G(A) + \int{\wh{1_A}(\gamma)d\gamma})/4 \geq \|1_A\|_{A(G)}/4,
\end{eqnarray*}
where the passage from the first to the second line is via the Fourier inversion formula.

We pick $M$ with $M = \log^{\Omega(1)} K_0$ such that
\begin{equation*}
\exp(-(9M)^{C_\mathcal{S}})>|K_0|^{-1/4} \textrm{ and } M \leq |K_0|^{1/8}/32,
\end{equation*}
where $C_\mathcal{S}>0$ is the absolute constant in Lemma \ref{lem.mainitlem}, and let $K':=K_0^{3/4}$ and $K := K_0^{3/8}/2^{11}M^2$.

Now, split into two cases; if $\|1_A\|_{A(G)} \geq M$ then we are done by our previous averaging argument, whence we shall assume that $\|1_A\|_{A(G)} \leq M$.

We construct a sequence of finite collections of subspaces $(\mathcal{H}_i)_i$ with
\begin{equation*}
|H| \geq K', |H \setminus A| \leq K \textrm{ and } \sup_{x \not \in H}{1_A \ast \mu_H(x)} \leq 1/2 \textrm{ for all } H \in \mathcal{H}_i
\end{equation*}
such that
\begin{equation*}
|H \cap H'| \leq K \textrm{ for all distinct pairs }H,H' \in \mathcal{H}_i.
\end{equation*}
We initialize with $\mathcal{H}_0=\emptyset$ which trivially satisfies the above and apply Lemma \ref{lem.mainitlem} repeatedly. If $i \leq 8\|1_A\|_{A(G)}$ then we see that 
\begin{equation*}
K_0 > 2^4|\mathcal{H}|^2K \textrm{ and } K' < \exp(-(\|1_A\|_{A(G)}+|\mathcal{H}|)^{C_\mathcal{S}})K_0,
\end{equation*}
whence each application of the lemma either tells us that $M_G(A) = K_0^{\Omega(1)}$ (and we are done) or that there is a new subgroup $H_0$ which may be added to $\mathcal{H}_i$ to get $\mathcal{H}_{i+1}$ thus blessed with all the desired properties.

However, it turns out that the iteration must terminate before this stage as we shall now see. Suppose that $H \in \mathcal{H}_i$. Then $|H \setminus A| \leq K$, whence
\begin{equation*}
\int_{\gamma \in H^\perp}{\wh{1_A}(\gamma)d\gamma} = \langle 1_A, \mu_H \rangle \geq 1-K/|H|
\end{equation*}
by Plancherel's theorem. Now let $H' \in \mathcal{H}_i$ have $H' \neq H$. Then
\begin{eqnarray*}
 \int{\wh{1_A}(\gamma)1_{(H+H')^\perp}(\gamma)d\gamma} & = & \langle1_A , \mu_{H+H'}\rangle\\ & = & \E_{W \in H+H'/H}{1_A \ast \mu_H(V)},
\end{eqnarray*}
which is well defined since $W$ is a coset of $H$ and $1_A \ast \mu_H$ is constant on cosets of $H$. It follows that
\begin{equation*}
| \int{\wh{1_A}(\gamma)1_{(H+H')^\perp}(\gamma)d\gamma}| \leq \frac{|H|}{|H+H'|}\left(1+\frac{1}{16\|1_A\|_{A(G)}}.\frac{|H+H'| - |H|}{|H|}\right).
\end{equation*}
On the other hand $|H+H'| = |H||H'|/|H \cap H'| \geq (K')^2/K$ whence
\begin{equation*}
| \int{\wh{1_A}(\gamma)1_{(H+H')^\perp}(\gamma)d\gamma}| \leq 1/16\|1_A\|_{A(G)}.
\end{equation*}
Now
\begin{equation*}
\int_{\gamma \in H^\perp}{\wh{1_A}(\gamma)d\gamma} -\int{\wh{1_A}(\gamma)1_{H^\perp \setminus (H')^\perp}(\gamma)d\gamma} = \int{\wh{1_A}(\gamma)1_{(H+H')^\perp}(\gamma)d\gamma},
\end{equation*}
whence
\begin{equation*}
\int_{\gamma \in H^\perp\setminus (H')^\perp}{|\wh{1_A}(\gamma)|d\gamma} \geq 1 - K/K' - 1/16\|1_A\|_{A(G)},
\end{equation*}
and writing $S_H=\bigcup\{H'^\perp: H' \in \mathcal{H}, H' \neq H\}$ we get from the triangle inequality that
\begin{equation*}
\int_{\gamma \in H^\perp\setminus S_H}{|\wh{1_A}(\gamma)|d\gamma} \geq 1 - 8\|1_A\|_{A(G)}K/K' - 1/2 \geq 1/4.
\end{equation*}
On the other hand the sets $(H\setminus S_H)_{H \in \mathcal{H}}$ are disjoint by design and so
\begin{equation*}
\|1_A\|_{A(G)} \geq \sum_{H \in \mathcal{H}}{\int_{\gamma \in H^\perp\setminus S_H}{|\wh{1_A}(\gamma)|d\gamma}} \geq |\mathcal{H}|/4.
\end{equation*}
It follows that in fact $|\mathcal{H}| \leq 4\|1_A\|_{A(G)}$ and the iteration terminates. The theorem is proved.
\end{proof}

\section{Concluding remarks}\label{sec.conrem}

As noted in the introduction lower bounds on the algebra norm of a set can be converted into lower bounds for $M_G$ by averaging. In view of this it is natural to take the quantitative idempotent theorem of \cite{BJGTS2} and try to derive a version of Theorem \ref{thm.densityindependentmaintheorem}.
\begin{theorem}[Quantitative idempotent theorem, {\cite[Theorem 1.3]{BJGTS2}}]\label{thm.bjg}
Suppose that $G$ is an abelian group and $A \subset G$ is a finite set. Then we may write
\begin{equation*}
1_A=\sum_{i=1}^L{\pm{1_{x_j+H_j}}}
\end{equation*}
where the $H_j \leq G$ are subgroups and $L = \exp(\exp(O(\|1_A\|_{A(G)}^4)))$. Moreover, the number of distinct subgroups $H_j$ is at most $M+O(1)$.
\end{theorem}
Of course doing this would require some work (most likely of the type in \S\ref{sec.trivial}) to take the structure produced by this theorem and convert it into the stronger output of Theorem \ref{thm.densityindependentmaintheorem} and in any case the most one could hope for would be doubly logarithmic bounds. 

Our proof proceeds in a rather different manner from that in \cite{BJGTS2} because we are unable to make use of almost boolean functions in the main iteration. This is largely because Chowla's problem is even more sensitive to changing sets into functions than the idempotent theorem is, and we have to proceed in a correspondingly more delicate way.

If one had the conjectural best possible version of Theorem \ref{thm.bjg} (where one is allowed to take $L = \exp(O(\|1_A\|_{A(G)}))$) one might hope to recover a lower bound of $\Omega(\log |A \triangle \bigcup{\mathcal{H}}|)$ in Theorem \ref{thm.densityindependentmaintheorem}. Of course one expects the bound to be much stronger and the following is really the interesting question.
\begin{problem}
Show that there is a function $\omega$ with $\omega(N) \rightarrow \infty$ as $N \rightarrow \infty$ such that for every non-empty, finite symmetric set $A \subset G$ there is a set $\mathcal{H}$ of subgroups of $G$ with $|\mathcal{H}| =O(M_G(A))$ such that
\begin{equation*}
M_G(A) = \Omega(\omega(|A \triangle \bigcup{\mathcal{H}}|)\log |A \triangle \bigcup{\mathcal{H}}|).
\end{equation*}
\end{problem}

On a more technical note it is possible to avoid the use of Bourgain systems by working heavily with the large spectrum. Doing this results in a doubly logarithmic bound for Theorem \ref{thm.densityindependentmaintheorem} because the Fourier space analogue of Proposition \ref{prop.weakfrei} is not very efficient at encoding the very large correlation that a set $A$ has with the associated Bourgain system. In any case, proceeding in this manner does not seem to be of any real benefit.

To close we remark that a number of related questions about the magnitude and arguments of various Fourier modes have been considered in the papers \cite{SVKVFL2} and \cite{SVKVFL1} of Lev and Konyagin. Interestingly, while our work is very analytic the obstacles in these papers become increasingly algebraic; in \cite{SVKVFL2}, for example, the properties of norms of algebraic integers are used.

\section*{Acknowledgements}

The author would like to thank Jean Bourgain for useful conversations and alerting us to the papers \cite{PEGS} and \cite{JPBPBBLBR}, Ben Green for useful conversations, Seva Lev for bringing the paper \cite{SVKVFL1} to our attention, and the anonymous referee for many useful comments and corrections.

\bibliographystyle{alpha}

\bibliography{master}

\end{document}